\documentclass[a4paper]{amsart}
\usepackage{amscd}
\usepackage{amsmath}
\usepackage{amssymb}
\usepackage{mathrsfs}
\usepackage{amsthm}
\usepackage{bbm} 
\usepackage{stmaryrd}
\usepackage{tikz-cd}

\usepackage[T1]{fontenc}

\numberwithin{equation}{section} \swapnumbers

\newtheorem{satz}{Satz}[section]

\newtheorem{theorem}[satz]{Theorem}
\newtheorem{proposition}[satz]{Proposition}
\newtheorem{corollary}[satz]{Corollary}
\newtheorem{lemma}[satz]{Lemma}

\newtheorem{definition}[satz]{Definition}

\newtheorem{remark}[satz]{Remark}

\newcommand{\bbr}{\mathbb{R}}
\newcommand{\bbe}{\mathbb{E}}
\newcommand{\bbn}{\mathbb{N}}
\newcommand{\bbp}{\mathbb{P}}
\newcommand{\bbq}{\mathbb{Q}}

\newcommand{\calb}{\mathscr{B}}

\newcommand{\calf}{\mathscr{F}}

\newcommand{\Var}{{\rm Var}}

\newcommand{\conv}{{\rm conv}}

\newcommand{\bbI}{\mathbbm{1}}

\begin{document}

\title[A note on the von Weizs\"{a}cker theorem]{A note on the von Weizs\"{a}cker theorem}
\author{Stefan Tappe}
\address{Karlsruhe Institute of Technology, Institute of Stochastics, Postfach 6980, 76049 Karlsruhe, Germany}
\email{jens-stefan.tappe@kit.edu}
\date{24 August, 2020}
\begin{abstract}
The von Weizs\"{a}cker theorem states that every sequence of nonnegative random variables has a subsequence which is Ces\`{a}ro convergent to a nonnegative random variable which might be infinite. The goal of this note is to provide a description of the set where the limit is finite. For this purpose, we use a decomposition result due to Brannath and Schachermayer.
\end{abstract}

\maketitle\thispagestyle{empty}

\section{Introduction}

Koml\'{o}s's theorem states that every $L^1$-bounded sequence has a subsequence which is Ces\`{a}ro convergent to a finite limit; see \cite{Komlos} and \cite{Berkes}. The paper \cite{W} was dealing with the question whether the $L^1$-boundedness can be dropped. Its main result states that every nonnegative sequence in $L^0$ has a subsequence which is Ces\`{a}ro convergent, but the limit can be infinite. More precisely, we have the following result.

\begin{theorem}[von Weizs\"{a}cker]\label{thm-W}
Let $(\xi_n)_{n \in \bbn} \subset L_+^0$ be a sequence of nonnegative random variables. Then there exist a subsequence $(\xi_{n_k})_{k \in \bbn}$ and a nonnegative random variable $\xi : \Omega \to [0,\infty]$ such that the following statements are true:
\begin{enumerate}
\item[(a)] The subsequence $(\xi_{n_{k}})_{k \in \bbn}$ is $\bbp$-almost surely Ces\`{a}ro convergent to $\xi$.

\item[(b)] There exists an equivalent probability measure $\bbq \approx \bbp$ such that the sequence $(\xi_{n_k} \bbI_{\{ \xi < \infty \}})_{k \in \bbn}$ is $L^1(\bbq)$-bounded.
\end{enumerate}
In addition, we can choose the subsequence $(\xi_{n_k})_{k \in \bbn}$ such that even one of the following stronger conditions is fulfilled:
\begin{enumerate}
\item[(a')] For every permutation $\pi : \bbn \to \bbn$ the sequence $(\xi_{n_{\pi(k)}})_{k \in \bbn}$ is $\bbp$-almost surely Ces\`{a}ro convergent to $\xi$.

\item[(a'')] For every further subsequence $(n_{k_l})_{l \in \bbn}$ the sequence $(\xi_{n_{k_l}})_{l \in \bbn}$ is $\bbp$-almost surely Ces\`{a}ro convergent to $\xi$.
\end{enumerate}
\end{theorem}

The existence of a subsequence such that (a') and (b) are fulfilled follows from \cite{W}. Slightly modifying its proof by using Koml\'{o}s's theorem from \cite{Komlos} (rather than the version from \cite{Berkes}) gives us the existence of a subsequence such that (a'') and (b) are fulfilled; see also \cite[Thm. 5.2.3]{Kabanov-Safarian} for such a result.

Let $(\xi_{n_k})_{k \in \bbn}$ be a subsequence as in Theorem~\ref{thm-W}; that is, conditions (a) and (b) are fulfilled. Note that the limit $\xi$ can be infinite. In this note we will provide a description of the sets $\{ \xi < \infty \}$ and $\{ \xi = \infty \}$. In \cite{W} it is already indicated that $\{ \xi < \infty \}$ should be the largest subset on which $(\xi_{n_k})_{k \in \bbn}$ is $L^1(\bbq)$-bounded for some equivalent probability measure $\bbq \approx \bbp$. However, \`{a} priori it is not clear whether such a set exists. We will approach this problem by looking at sets which are bounded in probability, without performing a measure change. For this purpose, we will use a decomposition result from \cite{B-Schach} which states that for every convex subset $C \subset L_+^0$ there exists a partition $\{ \Omega_b, \Omega_u \}$ such that $C|_{\Omega_b}$ is bounded in probability and $C$ is hereditarily unbounded in probability on $\Omega_u$. We refer to Section \ref{sec-proof} for the precise definitions. The set $\Omega_b$ is characterized as the largest subset on which the convex set $C$ is bounded in probability. 

Now, we define the convex hulls $C,\bar{C} \subset L_+^0$ as
\begin{align}\label{conv-hulls}
C := \conv \, \{ \xi_{n_k} : k \in \bbn \} \quad \text{and} \quad
\bar{C} := \conv \, \{ \bar{\xi}_{n_k} : k \in \bbn \},
\end{align}
where
\begin{align*}
\bar{\xi}_{n_k} := \frac{1}{k} \sum_{l=1}^k \xi_{n_l} \quad \text{for each $k \in \bbn$,}
\end{align*}
and denote by $\{ \Omega_b, \Omega_u \}$ and $\{ \bar{\Omega}_b, \bar{\Omega}_u \}$ the corresponding partitions according to the decomposition result from \cite{B-Schach}. Note that $\bar{C} \subset C$, because
\begin{align*}
\{ \bar{\xi}_{n_k} : k \in \bbn \} \subset C.
\end{align*}
Our result reads as follows.

\begin{proposition}\label{prop-main}
We have $\{ \xi < \infty \} = \Omega_b = \bar{\Omega}_b$ and $\{ \xi = \infty \} = \Omega_u = \bar{\Omega}_u$ up to $\bbp$-null sets.
\end{proposition}

\begin{remark}\label{rem-main}
Let us mention some further consequences:
\begin{enumerate}
\item The set $\{ \xi < \infty \}$ is the largest subset on which the convex hull of $(\xi_{n_k})_{k \in \bbn}$ or $(\bar{\xi}_{n_k})_{k \in \bbn}$ is bounded in probability.

\item The set $\{ \xi < \infty \}$ is also the largest subset on which the sequence $(\xi_{n_k})_{k \in \bbn}$ or $(\bar{\xi}_{n_k})_{k \in \bbn}$ is $L^1(\bbq)$-bounded for some equivalent probability measure $\bbq \approx \bbp$. This is in accordance with the findings in \cite{W}.

\item The set $\{ \xi < \infty \}$ is also the largest subset on which every sequence of convex combinations of $(\xi_{n_k})_{k \in \bbn}$ or $(\bar{\xi}_{n_k})_{k \in \bbn}$ has a convergent subsequence in the sense of weak convergence of their distributions.
\end{enumerate}
\end{remark}

For more details, we refer to Section \ref{sec-proof}, and in particular Corollary \ref{cor-finite}, where we investigate when the limit $\xi$ is almost surely finite. Note that Proposition \ref{prop-main} provides a link between the set $\{ \xi < \infty \}$ and the partition $\{ \Omega_b,\Omega_u \}$, which can be used in two directions; in particular, in certain situations we can conclude that the sets $C$ and $\bar{C}$ are bounded in probability. Suppose we have given a subsequence $( \xi_{n_k} )_{k \in \bbn}$ as in the von Weizs\"{a}cker theorem (Theorem \ref{thm-W}), and suppose we know the partition $\{ \Omega_b,\Omega_u \}$. Then we can easily determine the set $\{ \xi < \infty \}$; this is illustrated in Section \ref{sec-atomic} in the situation where we have an atomic probability space. Conversely, suppose we have given a subsequence $( \xi_{n_k} )_{k \in \bbn}$ as in the von Weizs\"{a}cker theorem (Theorem \ref{thm-W}), and suppose we know the set $\{ \xi < \infty \}$. Then we can easily determine the partitions $\{ \Omega_b,\Omega_u \}$ and $\{ \bar{\Omega}_b,\bar{\Omega}_u \}$. For illustration, we will assume in Section \ref{sec-SLLN} that the sequence $( \xi_n )_{n \in \bbn}$ satisfies the strong law of large numbers (SLLN). As we will see, then we can take the original sequence $(\xi_n)_{n \in \bbn}$ in the von Weizs\"{a}cker theorem; that is, we do not have to pass to a subsequence $(\xi_{n_k})_{k \in \bbn}$. As a consequence, the sets $C$ and $\bar{C}$ are given by
\begin{align*}
C = \conv \, \{ \xi_n : n \in \bbn \} \quad \text{and} \quad
\bar{C} = \conv \, \{ \bar{\xi}_n : n \in \bbn \},
\end{align*}
and we will provide a criterion when these sets are bounded in probability.

\section{Proof of the result}\label{sec-proof}

Let $(\Omega,\calf,\bbp)$ be a probability space. We denote by $L^0 = L^0(\Omega,\calf,\bbp)$ the space of all equivalence classes of random variables, where two random variables $X$ and $Y$ are identified if $\bbp(X=Y) = 1$. We denote by $L_+^0 = L_+^0(\Omega,\calf,\bbp)$ the convex cone of all nonnegative random variables; that is, all $X \in L^0$ such that $\bbp(X \geq 0) = 1$. It is well-known that $L^0$ equipped with the translation invariant metric
\begin{align*}
d(X,Y) = \bbe[|X-Y| \wedge 1], \quad X,Y \in L^0
\end{align*}
is a complete topological vector space. The induced convergence is just convergence in probability; that is, for a sequence $(X_n)_{n \in \bbn} \subset L^0$ and a random variable $X \in L^0$ we have $d(X_n,X) \to 0$ if and only if $X_n \overset{\bbp}{\to} X$. Furthermore, for every equivalent probability measure $\bbq \approx \bbp$ the translation invariant metric
\begin{align*}
d_{\bbq}(X,Y) = \bbe_{\bbq}[|X-Y| \wedge 1], \quad X,Y \in L^0
\end{align*}
induces the same topology.

\begin{definition}
A subset $C \subset L^0$ is called \emph{bounded in probability} (or \emph{$\bbp$-bounded}) if for every $\epsilon > 0$ there exists $M > 0$ such that
\begin{align*}
\sup_{X \in C} \bbp( |X| > M ) < \epsilon.
\end{align*}
\end{definition}

\begin{remark}
It is well-known that a subset $C \subset L^0$ is topologically bounded if and only if it is bounded in probability.
\end{remark}

\begin{remark}\label{rem-tight}
Note that a subset $C \subset L^0$ is bounded in probability if and only if the family of distributions $\{ \bbp \circ X : X \in C \}$ is tight.
\end{remark}

For a subset $C \subset L^0$ and an event $B \in \calf$ we agree on the notation
\begin{align*}
C|_B := \{ X \bbI_B : X \in C \}
\end{align*}

\begin{definition}
A subset $C \subset L^0$ is called \emph{hereditarily unbounded in probability} (or \emph{hereditarily $\bbp$-unbounded}) on a set $A \in \calf$ if for every $B \in \calf$ with $B \subset A$ and $\bbp(B) > 0$ the set $C|_B$ is not bounded in probability.
\end{definition}

\begin{definition}
A subset $C \subset L^0$ is called \emph{$L^1(\bbp)$-bounded} if
\begin{align*}
\sup_{X \in C} \bbe[|X|] < \infty.
\end{align*}
\end{definition}

The following two results which be useful for our analysis.

\begin{lemma}\cite[Lemma 2.3]{B-Schach}\label{lemma-B-Schach}
Let $C \subset L_+^0$ be a convex subset of $L_+^0$. Then there exists a partition $\{ \Omega_u,\Omega_b \}$ of $\Omega$ into disjoint sets $\Omega_u, \Omega_b \in \calf$, unique up to $\bbp$-null sets, such that:
\begin{enumerate}
\item $C|_{\Omega_b}$ is bounded in probability.

\item $C$ is hereditarily unbounded in probability on $\Omega_u$.
\end{enumerate}
\end{lemma}

In particular, for every event $B \in \calf$ such that $C|_B$ is bounded in probability, we have $B \subset \Omega_b$ up to $\bbp$-null sets.

\begin{lemma}\cite[Prop. 1.16]{Kostas}\label{lemma-conv-bounded}
Let $K \subset L_+^0$ be a subset. Then the following statements are equivalent:
\begin{enumerate}
\item[(i)] $\conv \, K$ is bounded in probability.

\item[(ii)] There exists an equivalent probability measure $\bbq \approx \bbp$ such that $K$ is $L^1(\bbq)$-bounded.
\end{enumerate}
\end{lemma}

\begin{remark}
The implication (i) $\Rightarrow$ (ii) also follows from \cite[Lemma 2.3.3]{B-Schach}. There it is even shown that one can find such an equivalent probability measure $\bbq \approx \bbp$ with bounded Radon-Nikodym density. 
\end{remark}

Now, let $(\xi_{n_k})_{k \in \bbn}$ be a subsequence as in Theorem \ref{thm-W}. Furthermore, let $C, \bar{C} \subset L_+^0$ be the convex sets given by (\ref{conv-hulls}). We denote by $\{ \Omega_u,\Omega_b \}$ and $\{ \bar{\Omega}_u,\bar{\Omega}_b \}$ the corresponding partitions according to Lemma \ref{lemma-B-Schach}.

\begin{lemma}\label{lemma-1}
We have $\Omega_b \subset \bar{\Omega}_b$ up to $\bbp$-null sets.
\end{lemma}

\begin{proof}
Since $\bar{C} \subset C$ and $C|_{\Omega_b}$ is bounded in probability, the set $\bar{C}|_{\Omega_b}$ is also bounded in probability, and hence we have $\Omega_b \subset \bar{\Omega}_b$.
\end{proof}

\begin{lemma}\label{lemma-2}
We have $\bar{\Omega}_b \subset \{ \xi < \infty \}$ up to $\bbp$-null sets.
\end{lemma}

\begin{proof}
We define the sequence $(\bar{\mu}_k)_{k \in \bbn}$ of probability measures on $(\bbr,\calb(\bbr))$ as $\bar{\mu}_k := \bbp \circ (\bar{\xi}_{n_k} \bbI_{\bar{\Omega}_b})$ for each $k \in \bbn$. Since $\bar{C}|_{\bar{\Omega}_b}$ is $\bbp$-bounded, by Remark \ref{rem-tight} the sequence $(\bar{\mu}_k)_{k \in \bbn}$ is tight. By Prohorov's theorem there is a subsequence $(\bar{\mu}_{k_l})_{l \in \bbn}$ such that $\bar{\mu}_{k_l} \overset{w}{\to} \mu$ for some probability measure $\mu$ on $(\bbr,\calb(\bbr))$. On the other hand, since $\bar{\xi}_{n_{k_l}} \bbI_{\bar{\Omega}_b} \overset{\text{a.s.}}{\to} \xi \bbI_{\bar{\Omega}_b}$, we have $\mu = \bbp \circ (\xi \bbI_{\bar{\Omega}_b})$. Therefore, it follows that $\bbp$-almost surely $\xi < \infty$ on $\bar{\Omega}_b$.
\end{proof}

\begin{lemma}\label{lemma-3}
We have $\{ \xi < \infty \} \subset \Omega_b$ up to $\bbp$-null sets.
\end{lemma}

\begin{proof}
By part (b) of Theorem \ref{thm-W} there exists an equivalent probability measure $\bbq \approx \bbp$ such that $(\xi_{n_k} \bbI_{\{ \xi < \infty \}})_{k \in \bbn}$ is $L^1(\bbq)$-bounded. Therefore, by Lemma \ref{lemma-conv-bounded} the set $C|_{\{ \xi < \infty \}}$ is $\bbp$-bounded, completing the proof.
\end{proof}

Now, the proof of Proposition \ref{prop-main} is a consequence of Lemmas \ref{lemma-1}--\ref{lemma-3}, and Remark \ref{rem-main} follows by additionaly taking into account Lemma \ref{lemma-conv-bounded} and Prohorov's theorem. 

Using Proposition \ref{prop-main} we can characterize when the limit $\xi$ is almost surely finite, and when it is almost surely infinite. As a consequence of the next result, the limit is almost surely finite if and only if every sequence of convex combinations has a convergent subsequence in the sense of weak convergence of their distributions.

\begin{corollary}\label{cor-finite}
The following statements are equivalent:
\begin{enumerate}
\item[(i)] We have $\xi < \infty$ almost surely.

\item[(ii)] The set $C$ is bounded in probability.

\item[(iii)] The set $\bar{C}$ is bounded in probability.

\item[(iv)] There exists an equivalent probability measure $\bbq \approx \bbp$ such that the sequence $(\xi_{n_k})_{k \in \bbn}$ is $L^1(\bbq)$-bounded.

\item[(v)] There exists an equivalent probability measure $\bbq \approx \bbp$ such that the sequence $(\bar{\xi}_{n_k})_{k \in \bbn}$ is $L^1(\bbq)$-bounded.

\item[(vi)] For every sequence $(\eta_n)_{n \in \bbn} \subset C$ there is a subsequence $(\eta_{n_k})_{k \in \bbn}$ such that $(\bbp \circ \eta_{n_k})_{k \in \bbn}$ converges weakly.

\item[(vii)] For every sequence $(\bar{\eta}_n)_{n \in \bbn} \subset \bar{C}$ there is a subsequence $(\bar{\eta}_{n_k})_{k \in \bbn}$ such that $(\bbp \circ \bar{\eta}_{n_k})_{k \in \bbn}$ converges weakly.
\end{enumerate}
\end{corollary}

\begin{proof}
This is an immediate consequence of Proposition \ref{prop-main}, Lemma \ref{lemma-conv-bounded} and Prohorov's theorem.
\end{proof}

\begin{corollary}\label{cor-infinite}
The following statements are equivalent:
\begin{enumerate}
\item[(i)] We have $\xi = \infty$ almost surely.

\item[(ii)] The set $C$ is hereditarily unbounded in probability.

\item[(iii)] The set $\bar{C}$ is hereditarily unbounded in probability.
\end{enumerate}
\end{corollary}

\begin{proof}
This is an immediate consequence of Proposition \ref{prop-main}.
\end{proof}

\section{Sequences on an atomic probability space}\label{sec-atomic}

In this section we assume that the probability space $(\Omega,\calf,\bbp)$ is atomic. More precisely, we assume there are subsets $(A_m)_{m \in \bbn}$ of $\Omega$ such that the following conditions are fulfilled:
\begin{itemize}
\item The sets $(A_m)_{m \in \bbn}$ are pairwise disjoint with $\Omega = \bigcup_{m \in \bbn} A_m$.

\item We have $\calf = \sigma(A_m : m \in \bbn)$.

\item We have $\bbp(A_m) > 0$ for all $m \in \bbn$.
\end{itemize}
Let $(\xi_n)_{n \in \bbn} \subset L_+^0$ be a sequence of nonnegative random variables. Furthermore, let $(\xi_{n_k})_{k \in \bbn}$ be a subsequence and let $\xi : \Omega \to [0,\infty]$ be a nonnegative random variable such that conditions (a) and (b) in the von Weizs\"{a}cker theorem (Theorem \ref{thm-W}) are fulfilled. We wish to determine the set $\{ \xi < \infty \}$. There are nonnegative numbers $(c_{n,m})_{n,m \in \bbn} \subset \bbr_+$ such that
\begin{align*}
\xi_n = \sum_{m=1}^{\infty} c_{n,m} \bbI_{A_m} \quad \text{for each $n \in \bbn$.}
\end{align*}
Let $J \subset \bbn$ be the set of all $m \in \bbn$ such that the sequence $(c_{n_k,m})_{k \in \bbn}$ is bounded, and let $I \subset J$ be the set of all $m \in \bbn$ such that the sequence $(c_{n,m})_{n \in \bbn}$ is bounded.

\begin{proposition}
We have $\bigcup_{m \in I} A_m \subset \{ \xi < \infty \} = \bigcup_{m \in J} A_m$ up to $\bbp$-null sets.
\end{proposition}

\begin{proof}
We only have to prove that $\{ \xi < \infty \} = \bigcup_{m \in J} A_m$ up to $\bbp$-null sets. We set $J_b := J$ and $J_u := \bbn \setminus J$, and the convex hull $C \subset L_+^0$ is defined by (\ref{conv-hulls}). By virtue of Proposition \ref{prop-main}, we have to show that $\Omega_b = \bigcup_{m \in J_b} A_m$ up to $\bbp$-null sets, where $\{ \Omega_b,\Omega_u \}$ denotes the partition from Lemma \ref{lemma-B-Schach}. For this purpose, we set $U_b := \bigcup_{m \in J_b} A_m$ and $U_u := \bigcup_{m \in J_u} A_m$. First, we show that $C|_{U_b}$ is bounded in probability. Note that
\begin{align*}
C|_{U_b} = \conv \{ \xi_{n_k} \bbI_{U_b} : k \in \bbn \}.
\end{align*}
For each $m \in J_b$ there is a constant $C_m \geq 1$ such that $c_{n_k,m} \leq C_m$ for each $k \in \bbn$. We define the sequence $(q_m)_{m \in \bbn} \subset (0,\infty)$ as
\begin{align*}
q_m := \frac{2^{-m}}{C_m}, \quad m \in I_b,
\\ q_m := 2^{-m}, \quad m \in I_u.
\end{align*}
By the geometric series we have
\begin{align*}
K := \sum_{m \in \bbn} q_m < \infty,
\end{align*}
and hence we can define the equivalent probability measure $\bbq \approx \bbp$ as
\begin{align*}
\bbq(A_m) := \frac{q_m}{K}, \quad m \in \bbn.
\end{align*}
Then by the geometric series we have
\begin{align*}
\sup_{k \in \bbn} \bbe_{\bbq}[ \xi_{n_k} \bbI_{U_b} ] &= \sup_{k \in \bbn} \bbe_{\bbq} \Bigg[ \bigg( \sum_{m=1}^{\infty} c_{n_k,m} \bbI_{A_m} \bigg) \cdot \bigg( \sum_{m \in J_b} \bbI_{A_m} \bigg) \Bigg]
\\ &= \sup_{k \in \bbn} \bbe_{\bbq} \bigg[ \sum_{m \in J_b} c_{n_k,m} \bbI_{A_m} \bigg] = \sup_{k \in \bbn} \sum_{m \in I_b} c_{n_k,m} \bbq(A_m)
\\ &\leq \sum_{m \in I_b} C_m \bbq(A_m) = \frac{1}{K} \sum_{m \in I_b} 2^{-m} < \infty.
\end{align*}
By Lemma \ref{lemma-conv-bounded} this shows that $C|_{U_b}$ is bounded in probability. Next, we will prove that $C|_{U_u}$ is hereditarily unbounded in probability. For this purpose, let $\bbq \approx \bbp$ be an arbitrary equivalent probability measure. Then we have $\bbq(A_m) > 0$ for all $m \in \bbn$. Therefore, for each $m \in J_u$ we obtain
\begin{align*}
\sup_{k \in \bbn} \bbe_{\bbq}[ \xi_{n_k} \bbI_{A_m} ] = \sup_{k \in \bbn} c_{n_k,m} \bbq(A_m) = \infty.
\end{align*}
Hence, by Lemma \ref{lemma-conv-bounded} the set $C|_{U_u}$ is hereditarily unbounded in probability, which completes the proof.
\end{proof}

\section{Sequences satisfying the strong law of large numbers}\label{sec-SLLN}

As in the previous sections, let $(\xi_n)_{n \in \bbn} \subset L_+^0$ be a sequence of nonnegative random variables. We say that $( \xi_n )_{n \in \bbn}$ satisfies the strong law of large numbers (SLLN) if there is a nonnegative, possibly infinite constant $\mu \in [0,\infty]$ such that the following conditions are fulfilled:
\begin{enumerate}
\item We have $\bbe[\xi_n] = \mu$ for all $n \in \bbn$.

\item $( \xi_n )_{n \in \bbn}$ is $\bbp$-almost surely Ces\`{a}ro convergent to $\mu$.
\end{enumerate}

If the sequence $( \xi_n )_{n \in \bbn}$ satisfies the SLLN, then conditions (a) and (b) in the von Weizs\"{a}cker theorem (Theorem \ref{thm-W}) are fulfilled by choosing the original sequence $( \xi_n )_{n \in \bbn}$ and $\xi = \mu$. Note that (b) is satisfied, because $(\xi_n \bbI_{\{ \xi < \infty \}})_{n \in \bbn}$ is $L^1(\bbp)$-bounded. Therefore, the sets $C$ and $\bar{C}$ are given by
\begin{align*}
C = \conv \, \{ \xi_n : n \in \bbn \} \quad \text{and} \quad
\bar{C} = \conv \, \{ \bar{\xi}_n : n \in \bbn \}.
\end{align*}
The following result shows that the sets $C$ and $\bar{C}$ are either bounded in probability or hereditarily unbounded in probability.

\begin{proposition}\label{prop-SLLN}
Suppose that $( \xi_n )_{n \in \bbn}$ satisfies the SLLN. Then the following statements are equivalent:
\begin{enumerate}
\item[(i)] We have $\mu < \infty$.

\item[(ii)] We have $\xi_1 \in L^1$.

\item[(iii)] The set $C$ is bounded in probability.

\item[(iv)] The set $\bar{C}$ is bounded in probability.
\end{enumerate}
Moreover, the following statements are equivalent:
\begin{enumerate}
\item[(i)] We have $\mu = \infty$.

\item[(ii)] We have $\xi_1 \notin L^1$.

\item[(iii)] The set $C$ is hereditarily unbounded in probability.

\item[(iv)] The set $\bar{C}$ is hereditarily unbounded in probability.
\end{enumerate}
\end{proposition}

\begin{proof}
This is a consequence of Proposition \ref{prop-main}.
\end{proof}

Now, we consider some examples where $( \xi_n )_{n \in \bbn}$ satisfies the SLLN. 

\begin{proposition}
Suppose that one of the following conditions is fulfilled:
\begin{enumerate}
\item[(a)] The sequence $( \xi_n )_{n \in \bbn}$ is i.i.d. with $\xi_1 \in L^1$.

\item[(b)] The sequence $( \xi_n )_{n \in \bbn}$ is a $\phi$-mixing sequence of identically distributed random variables with $\xi_1 \in L^2$ such that
\begin{align*}
\sum_{n=1}^{\infty} \sqrt{\phi(n)} \, \frac{\ln n}{n} < \infty.
\end{align*}

\item[(c)] The sequence $( \xi_n )_{n \in \bbn}$ is a $\rho$-mixing sequence of identically distributed random variables with $\xi_1 \in L^2$ such that
\begin{align*}
\sum_{n=1}^{\infty} \rho(n) < \infty.
\end{align*}

\item[(d)] We have $(\xi_n)_{n \in \bbn} \subset L^2$, there exists $\mu \in \bbr_+$ such that $\bbe[\xi_n] = \mu$ for all $n \in \bbn$, we have
\begin{align*}
\sum_{n \in \bbn} \frac{\Var[\xi_n]}{n^2} < \infty,
\end{align*}
and there exists a constant $c \in (0,\infty)$ such that for all $N \in \bbn$ we have
\begin{align}\label{corr}
\Var \bigg[ \sum_{n=1}^N \xi_n \bigg] \leq c \sum_{n=1}^N \Var[\xi_n].
\end{align}
\end{enumerate}
Then the sets $C$ and $\bar{C}$ are bounded in probability.
\end{proposition}

\begin{proof}
In any of the situations (a)--(d) the sequence $(\xi_n)_{n \in \bbn}$ satisfies the SLLN. This is immediately clear for (a). In situation (b) this follows from \cite[Cor. 1]{Kuc} and noting that $\phi$ remains invariant under adding constants to the sequence; similarly in situation (c) this follows from \cite[Cor. 4]{Kuc} and noting that $\rho$ remains invariant under adding constants to the sequence. In Situation (d) we use \cite[Thm. 1]{Janisch}. Consequently, by Proposition \ref{prop-SLLN} the sets $C$ and $\bar{C}$ are bounded in probability.
\end{proof}

Note that condition (\ref{corr}) is satisfied if $\bbe[\xi_n \xi_m] \leq \mu^2$ for all $n,m \in \bbn$ with $n \neq m$. In this case the result also follows from \cite[Thm. 1]{Etemadi}.

\section*{Acknowledgement}

I am grateful to an anonymous referee for helpful comments and suggestions.

\end{document}